\documentclass[a4paper,12pt]{article}
\usepackage{comment}
\usepackage{cite}
\usepackage{amsmath}
\usepackage{amssymb}
\usepackage{amsfonts}
\usepackage[T1]{fontenc}
\usepackage[utf8]{inputenc}
\usepackage{graphicx}
\usepackage{fancyhdr}
\usepackage{float}
\usepackage{xcolor}
\usepackage{authblk}
\usepackage{mathrsfs}
\usepackage{empheq}
\usepackage[hyphens]{url}
\usepackage{hyperref} 
\usepackage[]{breakurl}
\usepackage[]{amsthm}
\usepackage{tikz}

\newtheorem{theorem}{Theorem}

\usepackage{geometry}
\newcommand{\binod}[2]{\left[\begin{array}{c}\!#1\!\\\!#2\!\end{array}\right]}

\begin{document}

\title{A sum rule for $r$-derangements obtained from the Cauchy product of exponential generating functions}

\author[$\dagger$]{Jean-Christophe {\sc Pain}$^{1,2,}$\footnote{jean-christophe.pain@cea.fr}\\
\small
$^1$CEA, DAM, DIF, F-91297 Arpajon, France\\
$^2$Universit\'e Paris-Saclay, CEA, Laboratoire Mati\`ere en Conditions Extr\^emes,\\ 
91680 Bruy\`eres-le-Ch\^atel, France
}

\maketitle

\begin{abstract}
We propose a sum rule for $r$-derangements (meaning that the elements are restricted to be in distinct cycles in the cycle decomposition) involving binomial coefficients. The identity, obtained using the Cauchy product of two exponential generating functions, generalizes a known relation for usual derangements.
\end{abstract}

\section{Introduction}

The number of derangements of a $n$-element set is the number of fixed point-free permutations of $n$ elements \cite{Apostol1974,Graham1994} and is given by
\begin{equation}
D(n)=n!\sum_{i=0}^{n}\frac{(-1)^i}{i!},
\end{equation}
or also, for $n\geq 1$:
\begin{equation}
D(n)=\Big\lvert\Big\lvert\frac{n!}{e}\Big\lvert\Big\lvert=\Big\lfloor\frac{n!}{e}+\frac{1}{2}\Big\rfloor=\left\lfloor {\frac {n!+1}{e}}\right\rfloor,
\end{equation}
where $\lfloor x\rfloor$ represents the integer part of $x$, and $||x||$ the nearest integer to $x$.

Now, let us consider a permutation, and its cycle decomposition. The fact that we only deal here with fixed point-free permutations means that in the considered permutation, any cycle is of length greater than one. Now, let us take a permutation of $n+r$ elements and restrict the first $r$ of these to be in distinct cycles, then we arrive at the definition of the so-called $r$-derangements \cite{Wang2017}. In other words, for $r$-derangements the elements are restricted to be in distinct cycles in the cycle decomposition. As an example $(184)(2937)(65)$ is a 2-derangement (since 1 and 2 do not share any common circle), but not a 3-derangement.

The $r$-derangements are given, for $n\geq r$, by
\begin{equation}
    D_r(n)=\sum_{j=r}^n(-1)^{n-j}\,\binom{j}{r}\,\frac{n!}{(n-j)!}.
\end{equation}
It can easily be checked that $D_1(n)=D(n+1)$, $D_r(r)=r!$ (with $r\geq 1$) and $D_r(r+1)=r(r+1)!$ (for $r\geq 2$). It was shown by Wang \emph{et al.} \cite{Wang2017} that for $n>2$ and $r>0$, one has
\begin{equation}
    D_r(n)=r\,D_{r-1}(n-1)+(n-1)\,D_r(n-2)+(n+r-1)\,D_r(n-1).
\end{equation}

\section{Cauchy product of generating functions}

Let $(u_n)_{n\geq 0}$ and $(v_n)_{n\geq 0}$ be real or complex sequences. It was proved by Mertens that, if the series 
\begin{equation}
U(z)=\sum _{n=0}^{\infty}u_{n}
\end{equation}
converges to $L_1$, and 
\begin{equation}
V(z)=\sum_{n=0}^{\infty}v_{n}
\end{equation}
converges to $L_2$, and if at least one of them converges absolutely, then their Cauchy product 
\begin{equation}
\left(\sum _{n=0}^{\infty }u_{n}\right)\left(\sum _{n=0}^{\infty }v_{n}\right)=\sum _{n=0}^{\infty }w_{n} 
\end{equation}
with
\begin{equation}\label{wn}
w_{n}=\sum _{k=0}^{n}u_{k}v_{n-k}
\end{equation}
converges to $L_1\times L_2$ \cite{Rudin1976}. The theorem is still valid in a Banach algebra. Note that in the case where the two sequences are convergent but not absolutely convergent, the Cauchy product is still Ces\`aro summable \cite{Hardy2000} and then
\begin{equation}
\lim_{N\rightarrow\infty}\frac{1}{N}\left(\sum _{n=1}^{N}\sum_{i=1}^{n}\sum_{k=0}^{i}u_{k}v_{i-k}\right)=L_1L_2.
\end{equation}
In the case of two powers series,
\begin{equation}\label{s1}
\mathscr{U}(z)=\sum _{n=0}^{\infty}u_{n}\,x^n
\end{equation}
and 
\begin{equation}\label{s2}
\mathscr{V}(z)=\sum_{n=0}^{\infty}v_{n}\,x^n,
\end{equation}
the Cauchy product reads
\begin{equation}
\left(\sum_{n=0}^{\infty }u_{n}x^{n}\right)\cdot \left(\sum _{n=0}^{\infty }v_{n}x^{n}\right)=\sum _{n=0}^{\infty }w_{n}x^{n}, 
\end{equation}
with $w_n$ still given by Eq. (\ref{wn}).

In the case of two generating functions \cite{Wilf1990}:
\begin{equation}
A(x)=\sum_{n=0}^{\infty}a_n\,\frac{x^n}{n!}
\end{equation}
and
\begin{equation}
B(x)=\sum_{n=0}^{\infty}b_n\,\frac{x^n}{n!}
\end{equation}
which corresponds to setting $u_n=a_n/n!$ and $v_n=b_n/n!$ in Eqs. (\ref{s1}) and (\ref{s2}), we have
\begin{align}
A(x)B(x)&=\left(\sum_{n=0}^{\infty }a_{n}\frac{x^{n}}{n!}\right)\left(\sum _{n=0}^{\infty }b_{n}\frac{x^{n}}{n!}\right)\\
&=\sum_{n=0}^{\infty}\sum _{k=0}^{n}\frac{1}{k!}\frac{1}{(n-k)!}\,a_k\,b_{n-k}\,x^n\\
&=\sum_{n=0}^{\infty}\sum _{k=0}^{n}\frac{n!}{k!(n-k)!}\,a_k\,b_{n-k}\,\frac{x^n}{n!}\\
&=\sum_{n=0}^{\infty}\left(\sum_{k=0}^n\binom{n}{k}\,a_k\,b_{n-k}\right)\frac{x^{n}}{n!}\\
\end{align}
and thus
\begin{equation}\label{solu1}
A(x)B(x)=C(x)=\sum_{n=0}^{\infty}c_n\,\frac{x^{n}}{n!}
\end{equation}
with
\begin{equation}\label{solu2}
c_n=\sum_{k=0}^n\binom{n}{k}\,a_k\,b_{n-k}.
\end{equation}

\section{Sum rule for $r$-derangements}

\begin{theorem}

For $n\geq r$, the $r$-derangements satisfy the sum rule
\begin{equation}
\sum_{k=0}^n\binom{n}{k}\,D_r(k)=n!\binom{n}{r}.
\end{equation}

\end{theorem}

\begin{proof}
    
The generating function for $r$-derangements is (see for instance Ref. \cite{Wang2017}):
\begin{equation}
A(x)=\sum_{n=0}^{\infty}\frac{D_r(n)}{n!}\,x^n=\frac{x^r\,e^{-x}}{(1-x)^{r+1}}.
\end{equation}
The generating function of the constant value 1 is
\begin{equation}
B(x)=\sum_{k=0}^{\infty}\frac{x^k}{k!}=e^x.
\end{equation}
Therefore,
\begin{equation}
A(x)B(x)=\frac{x^r}{(1-x)^{r+1}}.
\end{equation}
Let us apply formulas (\ref{solu1}) and (\ref{solu2}) with $a_k=D_r(k)$ and $b_{n-k}=1$; we have
\begin{equation}
c_n=\sum_{k=0}^n\binom{n}{k}\,a_k\,b_{n-k}=\sum_{k=0}^n\binom{n}{k}\,D_r(k)=n!\,[x^n]\,\sum_{k=0}^{\infty}\frac{(r+1)(r+2)\cdots(r+q)}{q!}\,x^{q+r}
\end{equation}
and keeping only the term $q+r=n$ in the summation of the right-hand side, one gets
\begin{equation}
\sum_{k=0}^n\binom{n}{k}\,D_r(k)=n!\,\frac{(r+q)!}{r!q!}=n!\,\frac{n!}{r!(n-r)!}
\end{equation}
and thus, finally
\begin{equation}
\sum_{k=0}^n\binom{n}{k}\,D_r(k)=n!\binom{n}{r},
\end{equation}
which completes the proof.

\end{proof}

The case of usual derangements is recovered setting $r=0$, and one has
\begin{equation} 
\sum_{k=0}^n\binom{n}{k}\,D(k)=n!.   
\end{equation}
It is worth mentioning that Mez\H{o} \emph{et al.} used orthogonal polynomials (in particular Charlier polynomials) to derive new identities for $r$-derangements \cite{Mezo2019}.

The technique presented here can be applied to other derangements \cite{Chow2006,Caparelli2018,Moll2018,Mezo2019b,Zhang2022,Mezo2021}, such as B-type derangements (OEIS sequence A000354) \cite{OEIS}. Let us consider the ensemble $\mathscr{E}_n=\{1,2,3,\cdots,n\}$. The number $D^B(n)$ of signed permutations $\sigma$ such that $\sigma(i)\ne i$ $\forall i\in \mathscr{E}_n$ is given by \cite{Chow2006,Assaf2010}:
\begin{equation}
    D^B(n)=n!\sum_{k=0}^n\frac{(-1)^k\,2^{n-k}}{k!}
\end{equation}
and satisfies the generating function
\begin{equation}
    \sum_{n=0}^{\infty}D^B(n)\,\frac{x^n}{n!}=\frac{e^{-x}}{1-2x}.
\end{equation}
Similarly, $r$-derangements of $B$-type on the set $\mathscr{E}_{n+r}$ are signed permutations of $\mathscr{E}_{n+r}$, without fixed points and with $r$ elements constrained to be in distinct cycles. Their number can be denoted $D_r^B(n)$ and is equal to \cite{Mezo2021}:
\begin{equation}
    D_r^B(n)=\sum_{k=0}^n\binod{n}{k}^B_{\geq 2,r},
\end{equation}
where $\binod{n}{k}^B_{\geq 2,r}$ is a so-called $r$-Stirling number of $B$ type. The latter is related to the Lah numbers $L(n_1,n_2)$ (quantifying the number of ways a set of $n_1$ elements can be partitioned into $n_2$ nonempty linearly ordered subsets\cite{Riordan2002}) through (see for instance Refs. \cite{Broder1984,Ma2023}):
\begin{equation}
    \binod{n}{0}^B_{\geq 2,r}=\sum_{j=0}^r\binom{r}{j}\,2^{n+r-j}\,(r-j)!\,L(n,r-j).
\end{equation}
The nice following sum rule
\begin{equation}
L(n,r-1)=\frac{1}{(r+1)!}\sum_{s=0}^{n-1}\binom{n}{s}(n-s)\,D_r(s)=\frac{1}{(r+1)!}\sum_{s=0}^{n}\binom{n}{s}\,s\,D_r(n-s)
\end{equation}
was recently obtained \cite{Wang2017}. The $B$-type $r$-Stirling numbers are generalizations of $r$-Stirling numbers, and are related to Riordan arrays \cite{Shapiro1991,Riordan2002,Barry2007} and Coxeter groups \cite{Bjorner2005}.

\section{Conclusion}

We obtained, using the Cauchy product of two exponential generating functions, a sum rule for $r$-derangements involving binomial coefficients. The basic idea consists in taking advantage of the fact that the exponential term $e^{-x}$ in the generating function of $r$-derangements cancels with the $e^x$ one in the generating function of the constant suite equal to unity. We plan to derive similar identities to generalized derangements, such as $B$-type $r$-derangements.

\end{document}